\documentclass[12pt,a4paper,reqno]{amsart}
\usepackage[english]{babel}
\usepackage[applemac]{inputenc}
\usepackage[T1]{fontenc}
\usepackage{palatino}
\usepackage{amsmath}
\usepackage{amssymb}
\usepackage{amsthm}
\usepackage{amsfonts}
\usepackage{graphicx}
\usepackage[colorlinks = true, citecolor = black]{hyperref}
\author{Tuomas Orponen}
\title[Ahlfors-David regular weighted bounds for the extension opertor]{On Ahlfors-David regular weighted bounds for the extension operator associated to the circle}
\address{School of Mathematics, University of Edinburgh}
\thanks{T.O. gratefully acknowledges the financial support of the Finnish foundation Jenny and Antti Wihurin Rahasto.}
\email{tuomas.orponen@helsinki.fi}

\newcommand{\R}{\mathbb{R}}
\newcommand{\N}{\mathbb{N}}

\newcommand{\Z}{\mathbb{Z}}

\newcommand{\calD}{\mathcal{D}}

\newcommand{\spt}{\operatorname{spt}}

\newcommand{\calE}{\mathcal{E}}

\newcommand{\diam}{\operatorname{diam}}

\newcommand{\sgn}{\operatorname{sgn}}
\newcommand{\Rea}{\operatorname{Re}}
\newcommand{\Ai}{\operatorname{Ai}}
\newcommand{\Bi}{\operatorname{Bi}}

\numberwithin{equation}{section}

\theoremstyle{plain}
\newtheorem{thm}[equation]{Theorem}

\newtheorem{lemma}[equation]{Lemma}

\newtheorem{proposition}[equation]{Proposition}

\theoremstyle{definition}

\theoremstyle{remark}

\begin{document}

\maketitle

\begin{abstract} This paper addresses the sharpness of a weighted $L^{2}$-estimate for the Fourier extension operator associated to the circle, obtained by J. Bennett, A. Carbery, F. Soria and A. Vargas in 2006. A point left open in their paper was the necessity of a certain $\log R$-factor in the bound. Here, I show that the factor is necessary for all $1/2$-Ahlfors-David regular weights on the circle, but it can be removed for $s$-Ahlfors-David regular weights with $s \neq 1/2$.
 
\end{abstract}

\section{Introduction} This paper is concerned with the Fourier extension operator $g \mapsto \widehat{g d\sigma}$ associated to the circle $S^{1} \subset \R^{2}$, defined for all $g \in L^{1}(S^{1})$ by
\begin{displaymath} \widehat{gd\sigma}(x) = \int_{S^{1}} e^{-2\pi i x \cdot \xi} g(\xi) \, d\sigma(\xi). \end{displaymath} 
Here $\sigma$ is the length measure on $S^{1}$. The following weighted inequality for $\widehat{\cdot d\sigma}$ was established by J. Bennett, A. Carbery, F. Soria and A. Vargas \cite{BCSV} in 2006:
\begin{equation}\label{bcsv} \int_{S^{1}} |\widehat{gd\sigma}(Rx)|^{2} \, d\mu x \lesssim \frac{\log R}{R}\|g\|_{2}^{2}\sup_{R^{-1} \leq \alpha \leq R^{-2/3}} \frac{\mu(T(\alpha,\alpha^{2}R))}{\alpha}. \end{equation}
By assumption, the measure $\mu$ is supported on $S^{1}$, and the notation $T(\alpha,\beta)$ stands for an arbitrary rectangle with dimensions $\alpha \times \beta$; so, the $\sup$ is taken over all the scales $R^{-1} \leq \alpha \leq R^{-2/3}$ and over all $(\alpha \times \alpha^{2}R)$-rectangles in $\R^{2}$, with any orientation. The quantity $\|g\|_{2}^{2}$ appearing on the right hand side of \eqref{bcsv} is the square of the \textbf{unweighted} $L^{2}$-norm of $g$ on $S^{1}$. As usual, the notation $A \lesssim B$ means that $A \leq CB$ for some absolute constant $C \geq 1$.

One motivation to study inequalities of the form \eqref{bcsv} stems from fact (see \cite[Proposition 1]{BCSV} or \cite{CSV}) that the best constant $A$ in the inequality
\begin{displaymath} \int_{\{|x| \leq 1\}} |\widehat{gd\sigma}(Rx)|^{2} \, d\mu x \lesssim A \cdot \frac{\|g\|_{2}^{2}}{R} \end{displaymath}
is comparable with the best constant $B$ in the inequality
\begin{equation}\label{multiplier} \int_{\{|x| \leq 1\}} |S^{1/R}f(x)|^{2} \, d\mu x \leq B \int_{\R^{2}} |f(x)|^{2} \, dx, \end{equation}
where $S^{1/R}$ is the Fourier multiplier in $\R^{2}$ with symbol in $\Phi(|\xi| - R)$, and $\Phi$ is a non-negative normalised smooth bump function in one variable.

The necessity of the $\log R$-factor in \eqref{bcsv} was left open in \cite{BCSV}, and the primary purpose of the present paper is to address this issue. By definition, a measure $\mu$ on $\R^{2}$ is \emph{$s$-Ahlfors-David regular}, if
\begin{equation}\label{ahlforsDavid} \mu(B(x,r)) \sim r^{s}, \qquad x \in \spt \mu, \: 0 < r \leq \diam(\spt \mu), \end{equation}
where $A \sim B$ is shorthand for $A \lesssim B \lesssim A$. The constant
\begin{displaymath} M_{R}(\mu) := \sup_{R^{-1} \leq \alpha \leq R^{-2/3}} \frac{\mu(T(\alpha,\alpha^{2}R))}{\alpha} \end{displaymath}
is readily computed for $s$-Ahlfors-David regular measures supported on $S^{1}$. Here are the numbers:
\begin{proposition}\label{constants} Let $\mu$ be an $s$-Ahlfors-David regular measure supported on $S^{1}$. Then,
\begin{displaymath} M_{R}(\mu) \sim \begin{cases} R^{(2 - s)/3}, & \text{if } 1/2 \leq s \leq 1,\\ R^{1 - s}, & \text{if } 0 \leq s \leq 1/2. \end{cases} \end{displaymath}
\end{proposition}
The point to observe is that the number $s = 1/2$ has a special role. Now, consider the quantity
\begin{displaymath} \calE(\mu,R) := \sup_{\|g\|_{2} = 1} R \cdot \int_{S^{1}} |\widehat{g d\sigma}(Rx)|^{2} \, d\mu x. \end{displaymath}
Inequality \eqref{bcsv} can be restated as $\calE(\mu,R) \lesssim M_{R}(\mu) \cdot \log R$. The main result of the paper, below, shows that the $(\log R)$-factor can be dispensed with for all $s$-Ahlfors-David regular measures with $s \neq 1/2$, but, on the other hand, the factor is necessary for \textbf{all} $1/2$-Ahlfors-David regular measures:
\begin{thm}\label{main} Let $0 \leq s \leq 1$, and let $\mu$ be an $s$-Ahlfors-David regular measure on $S^{1}$. Then, there exist arbitrarily large $R \geq 1$ such that
\begin{displaymath} \calE(\mu,R) \sim \begin{cases} M_{R}(\mu), & \text{if } s \neq \tfrac{1}{2},\\ M_{R}(\mu) \cdot \log R, & \text{if } s = \tfrac{1}{2}. \end{cases} \end{displaymath}
\end{thm}
Even if there is a slight improvement over \eqref{bcsv} for the upper bound in the case $s \neq \tfrac{1}{2}$, the proof is still very much the same as in \cite{BCSV}; so, this part of the result is included mainly to demonstrate the special role of $s = 1/2$ (nevertheless, a proof is included in Section \ref{sNeqHalf}). The upper bound in the case $s = \tfrac{1}{2}$ is precisely \eqref{bcsv}, so the main point of the whole paper is to establish the lower bound 
\begin{displaymath} \calE(\mu,R) \gtrsim M_{R}(\mu) \cdot \log R \sim R^{1/2} \cdot \log R \end{displaymath}
for arbitrarily large $R \geq 1$. Note that this also gives a lower bound for the best constant $B$ in the inequality \eqref{multiplier}, for all $1/2$-Ahlfors-David regular measures $\mu$ supported on the unit circle. 

\section{The cases $s \neq \tfrac{1}{2}$}\label{sNeqHalf}

The constants $M_{R}(\mu)$ in Proposition \ref{constants} are computed as follows:
\begin{proof}[Proof of Proposition \ref{constants}] Fix $R^{-1} \leq \alpha \leq R^{-2/3}$. Then, it is possible to choose a rectangle $T(\alpha,\alpha^{2}R)$ so that $T(\alpha,\alpha^{2}R)$ contains an arc $J \subset S^{1}$ of length $\ell(J) \sim \alpha^{2}R$ (to see this, consider first the "hardest" case $\alpha = R^{-2/3}$). It follows that
\begin{displaymath} \frac{\mu(T(\alpha,\alpha^{2}R))}{\alpha} \gtrsim \frac{\alpha^{2s}R^{s}}{\alpha} = \alpha^{2s - 1}R^{s}. \end{displaymath}
Now, depending on whether $s \leq 1/2$ or $s \geq 1/2$, the expression above is maximised by choosing either $\alpha = R^{-1}$ or $\alpha = R^{-2/3}$ -- and, of course, the expression is independent of $\alpha$ when $s = 1/2$. These choices give the lower bounds in Proposition \ref{constants}. The upper bounds are obtained by observing that $T(\alpha,\alpha^{2}R)$ is always contained in a ball of radius $\sim \alpha^{2}R$.
\end{proof}

Next, I sketch the proof of the upper bounds in Theorem \ref{main}:

\begin{proof}[Proof of the upper bounds in Theorem \ref{main}] The proof of \eqref{bcsv} in \cite{BCSV} is based on the following representation of the extension operator. Suppose that $g \in L^{1}(S^{1})$ is so smooth that
\begin{displaymath} g(e^{2\pi i \theta}) = \sum_{k \in \Z} a_{k}e^{2\pi i k \theta}, \qquad \theta \in [0,1]. \end{displaymath}
Then
\begin{equation}\label{representation} \widehat{g d\sigma}(Re^{2\pi i\theta}) = \sum_{k \in \Z} a_{k}J_{k}(R)e^{2\pi i k \theta}, \qquad \theta \in [0,1], \end{equation}
where $J_{k}$ is the $k^{th}$ Bessel function of the first kind. Next, assuming (momentarily) that the non-zero Fourier coefficients of $g$ are supported on the interval $\{R/2,\ldots,R\}$, one decomposes $g$ into $\sim \log R$ pieces $g_{p}$ such that the non-zero Fourier coefficients of $\widehat{g_{p}}(k)$ are supported on those indices $k$ with $R - k \sim 4^{p}R^{1/3}$, $1 \leq 4^{p} \lesssim R^{2/3}$; for a fixed $p$, denote the set of such indices by $A_{p}$ (as in \cite{BCSV}). At this point, the proof divides into the cases $s > 1/2$ and $s < 1/2$ (the case $s = 1/2$ being already covered in \cite{BCSV}).
\subsection{The case $s > 1/2$} Choosing a small constant $\beta > 0$ and using the linearity of the extension operator, one finds that
\begin{align} R \cdot \int_{S^{1}} |\widehat{gd\sigma}(Rx)|^{2} \, d\mu x & = R \cdot \int_{S^{1}} \left| \sum_{1 \leq 4^{p} \lesssim R^{2/3}} 2^{-\beta p}2^{\beta p} \widehat{g_{p} d\sigma}(Rx) \right|^{2} \, d\mu x \notag\\
&\label{form7} \lesssim_{\beta} \sum_{1 \leq 4^{p} \lesssim R^{2/3}} 2^{2\beta p} R \cdot \int_{S^{1}} |\widehat{g_{p}d\sigma}(Rx)|^{2} \, d\mu x. \end{align}
Next, using the representation \eqref{representation}, the integrals can be written as
\begin{equation}\label{form8} \int_{S^{1}} |\widehat{g_{p}d\sigma}(Rx)|^{2} \, d\mu x = \sum_{j,k \in A_{p}} a_{j}\overline{a_{k}}J_{j}(R)J_{k}(R)\hat{\mu}(j - k). \end{equation}
Since $|j - k| \lesssim 4^{p}R^{1/3}$ for $j,k \in A_{p}$, one then finds a smooth function $P_{p} \colon \R \to \R$ satisfying $\widehat{P_{p}}(j - k) = 1$ for all $j,k \in A_{p}$, and
\begin{displaymath} |P_{p}(t)| \lesssim_{N} \frac{4^{p}R^{1/3}}{(1 + 4^{p}R^{1/3}|t|)^{N}}, \qquad N \in \N. \end{displaymath}
Then
\begin{align*} \int_{S^{1}} |\widehat{g_{p}d\sigma}(Rx)|^{2} \, d\mu x & = \int_{S^{1}} |\widehat{g_{p}d\sigma}(Rx)|^{2} \mu \ast P_{p}(x) \, d\sigma x\\
& \leq \|\mu \ast P_{p}\|_{L^{\infty}} \int_{S^{1}} |\widehat{g_{p}d\sigma}(Rx)|^{2} \, d\sigma x \end{align*}
The first factor is bounded by
\begin{displaymath} \|\mu \ast P_{p}\|_{L^{\infty}} \lesssim 4^{p(1 - s)}R^{(1 - s)/3} \end{displaymath}
using the growth bound $\mu(B(x,r)) \lesssim r^{s}$ (or see Lemma \ref{convolutionEstimates} below for details), while for the second factor, one has (using \eqref{form8} and $\hat{\sigma}(j - k) = \delta_{j,k}$)
\begin{displaymath} \int_{S^{1}} |\widehat{g_{p}d\sigma}(Rx)|^{2} \, d\sigma x = \sum_{k \in A_{p}} |a_{k}|^{2}|J_{k}(R)|^{2} \lesssim \frac{2^{-p}}{R^{2/3}} \sum_{k \in A_{p}} |a_{k}|^{2} \leq \frac{2^{-p}}{R^{2/3}}\|g\|_{2}^{2}. \end{displaymath}
Here the uniform bound
\begin{displaymath} |J_{k}(R)| \lesssim R^{-1/2} \cdot \min\left\{k^{1/6},\left|\frac{|R| + |k|}{|R| - |k|}\right|^{1/4} \right\} \end{displaymath}
was used, see \cite[Lemma 5]{BCSV} (or use the techniques in the proof of Lemma \ref{average} below to deduce the result). So, all in all,
\begin{displaymath} R \cdot \int_{S^{1}} |\widehat{g_{p}d\sigma}(Rx)|^{2} \, d\mu x \lesssim R^{(2 - s)/3} \cdot 2^{p(1-2s)}\|g\|_{2}^{2} \sim M_{R}(\mu) \cdot 2^{p(1 - 2s)}\|g\|_{2}^{2}. \end{displaymath}
Thus, if $2\beta < 2s - 1$, one may infer the sum on line \eqref{form7} adds up to a constant times $M_{R}(\mu) \cdot \|g\|_{2}^{2}$, as desired. A similar argument takes care of functions $g$, whose Fourier support lies in $\{R,\ldots,3R/2\}$, $\{-R,\ldots,-R/2\}$ and $\{-3R/2,\ldots,-R\}$. The remaining cases, where $\spt \hat{g} \subset \{|k| > 3R/2\}$ or $\spt \hat{g} \subset \{|k| < R/2\}$, have already been dealt with in \cite[(2), Proposition 6]{BCSV}. This completes the proof in the case $s > 1/2$, because now any function $g$ can be split up into at most six pieces, each one of which has been handled separately above.

\subsection{The case $s < 1/2$} One proceeds almost as above, with the single difference that instead of the factors $2^{\beta p}$ and $2^{-\beta p}$, one introduces $2^{\beta p}/R^{\beta/3}$ and $(2^{\beta p}/R^{\beta/3})^{-1}$. Since
\begin{displaymath} \sum_{1 \leq 4^{p} \lesssim R^{2/3}} \left(\frac{2^{\beta p}}{R^{\beta/3}}\right)^{2} \lesssim_{\beta} 1, \qquad \beta > 0, \end{displaymath}
one obtains the following analogue of \eqref{form7}:
\begin{displaymath} R \cdot \int_{S^{1}} |\widehat{g d\sigma}(Rx)|^{2} \, d\mu x \lesssim_{\beta} \sum_{1 \leq 4^{p} \lesssim R^{2/3}} \left(\frac{2^{p}}{R^{1/3}} \right)^{-2\beta} R \cdot \int_{S^{1}} |\widehat{g_{p}d\sigma}(Rx)|^{2} \, d\mu x, \quad \beta > 0. \end{displaymath}
Next, the proof continues as in the case $s > 1/2$ until one has reached the estimate
\begin{displaymath} R \cdot \int_{S^{1}} |\widehat{g_{p}d\sigma}(Rx)|^{2} \, d\mu x \lesssim R^{(2 - s)/3} \cdot 2^{p(1-2s)}\|g\|_{2}^{2}. \end{displaymath}
The numbers here need to be interpreted as
\begin{displaymath} R^{(2 - s)/3} \cdot 2^{p(1-2s)} = R^{1 - s} \cdot \left(\frac{2^{p}}{R^{1/3}}\right)^{1 - 2s} \sim M_{R}(\mu) \cdot \left(\frac{2^{p}}{R^{1/3}}\right)^{1 - 2s}, \end{displaymath}
so that finally 
\begin{displaymath} R \cdot \int_{S^{1}} |\widehat{g d\sigma}(Rx)|^{2} \, d\mu x \lesssim M_{R}(\mu) \cdot \|g\|_{2}^{2} \sum_{1 \leq 4^{p} \lesssim R^{2/3}} \left(\frac{2^{p}}{R^{1/3}}\right)^{1 - 2s -2\beta} \lesssim_{\beta} M_{R}(\mu) \cdot \|g\|_{2}^{2}, \end{displaymath}
as long as $2\beta < 1 - 2s$. The rest of the proof is similar to the case $s > 1/2$.
\end{proof}

I postpone the discussion of the sharpness of the bounds until the end of the next section. 

\section{The case $s = \tfrac{1}{2}$}

In this section, $\mu$ is an $1/2$-Ahlfors-David regular probability measure on $[0,1] \cong S^{1}$, unless otherwise stated, and $R \geq 1$ is large. What follows is a construction of a function $g \in L^{2}(S^{1})$ with $\|g\|_{2} = 1$ such that for appropriately chosen radii $r \sim R$, one has
\begin{displaymath} \int_{S^{1}} |\widehat{g d\sigma}(r x)|^{2} \, d\mu x \gtrsim \frac{\log R}{R^{1/2}}\|g\|_{2}^{2}. \end{displaymath}
I recall from the previous section that if $g$ has the representation
\begin{displaymath} g(e^{2\pi i \theta}) = \sum_{k \in \Z} a_{k}e^{2\pi i k \theta}, \qquad \theta \in [0,1], \end{displaymath}
then
\begin{displaymath} \widehat{g d\sigma}(re^{2\pi i\theta}) = \sum_{k \in \Z} a_{k}J_{k}(r)e^{2\pi i k \theta}, \qquad \theta \in [0,1], \end{displaymath}
where $J_{k}$ is the $k^{th}$ Bessel function of the first kind. The construction of $g$ is based on these formulae, so one needs some understanding about the asymptotic behaviour of $J_{k}(r)$, for large $k$ and $r$. This is given by the next lemma:
\begin{lemma}\label{average} Assume that $R - k \sim 4^{p}R^{1/3}$, where $C \leq 4^{p} \leq R^{2/3}/C$ for some large enough absolute constant $C \geq 1$. Then
\begin{displaymath} A_{k}(R) := \frac{1}{R^{1/3}} \int_{R - R^{1/3}}^{R + R^{1/3}} |J_{k}(r)| \, dr \gtrsim \frac{1}{2^{p/2}R^{1/3}}. \end{displaymath}
\end{lemma}

\begin{proof} In brief, the point here is that when $R - k \sim 4^{p}R^{1/3}$ and $r \sim R$, the function $r \mapsto J_{k}(r)$ oscillates roughly between $-2^{-p/2}R^{-1/3}$ and $2^{-p/2}R^{-1/3}$. Moreover, for $p$ large enough, the frequency of the oscillation is so high that an interval of length $\sim R^{1/3}$ contains a "peak" of $r \mapsto J_{k}(r)$.

To make the argument precise, one needs a fair understanding of the asymptotic behaviour of $J_{k}(r)$, which is fortunately contained in Erd\'elyi's treatise \cite{E}. Namely, (10) on \cite[p. 107]{E} gives the asymptotic expansion
\begin{equation}\label{form4} J_{k}(k \lambda) = \left(\frac{\lambda}{2}k^{2/3} \phi'(\lambda) \right)^{-1/2}\Ai(-k^{2/3}\phi(\lambda))[1 + O(1/k)], \end{equation}
which holds (quoting Erdelyi) \emph{uniformly in $\lambda$, $0 < \lambda < \infty$, as $k \to \infty$, $\Rea k \geq 0$, except that the error term needs some modification near zeros of $\Ai(-k^{2/3}\phi(\lambda))$.} The function $\phi$ is defined as the unique solution to the differential equation
\begin{displaymath} \phi(\lambda)(\phi'(\lambda))^{2} = 1 - \frac{1}{\lambda^{2}}, \end{displaymath}
so that $\phi' > 0$, and $\phi'$ is bounded uniformly away from zero and infinity on $[1/2,3/2]$, see \cite[p. 98]{E}. The function $\Ai$ is the \emph{Airy function}, which for non-positive arguments has the relatively simple expression
\begin{equation}\label{airy} \Ai(-t) = \frac{\sqrt{t}}{3}\left(J_{\tfrac{1}{3}}\left(\tfrac{2}{3}x^{\tfrac{3}{2}}\right) + J_{-\tfrac{1}{3}}\left(\tfrac{2}{3}x^{\tfrac{3}{2}}\right) \right), \quad t \geq 0. \end{equation}

Some inconvenience is caused by the fact that the function $\Ai(-k^{2/3}\phi(\lambda))$ has, indeed, zeroes in the region relevant to the proof, so one has to get acquainted with the meaning of the "some modification" of the error term. This modification is shown in (15) of \cite[p. 102]{E} (for the convenience of the reader interested in tracking the reference, I mention that the functions $Y_{0}$ and $Y_{2}$ in this equation are defined in (3) of \cite[p. 98]{E}, whereas $y_{0}(\lambda)$ is related to $J_{k}(k\lambda)$ on the last line of \cite[p. 106]{E}). Decoding Erd\'elyi's notation, the end result looks like
\begin{align} J_{k}(k \lambda) = & \left(\frac{\lambda}{2}k^{2/3}\phi'(\lambda) \right)^{-1/2} \Ai(-k^{2/3}\phi(\lambda))[1 + O(1/k)] \notag\\
&\label{form5} \qquad + O(1/k)\left(\frac{\lambda}{2}k^{2/3}\phi'(\lambda) \right)^{-1/2}\Bi(-k^{2/3}\phi(\lambda)). \end{align}
So, in comparison with \eqref{form4}, there is the added term on line \eqref{form5}, where, for non-positive arguments,
\begin{displaymath} \Bi(-t) = \sqrt{\frac{t}{3}}\left(J_{-\tfrac{1}{3}}\left(\tfrac{2}{3}x^{\tfrac{3}{2}}\right) - J_{\tfrac{1}{3}}\left(\tfrac{2}{3}x^{\tfrac{3}{2}}\right) \right), \quad t \geq 0. \end{displaymath}
Fortunately, for $t \geq 0$, one has, $|Ai(-t)|, |\Bi(-t)| \leq C$ for some absolute constant $C$, so one can now deduce the weaker expansion
\begin{equation}\label{form6} J_{k}(k\lambda) = \left(\frac{\lambda}{2}k^{2/3}\phi'(\lambda) \right)^{-1/2} \Ai(-k^{2/3}\phi(\lambda)) + O(k^{-4/3}), \end{equation}
valid for $\lambda \in [1/2,3/2]$ and for large enough $k \geq 0$. In particular, with $\lambda = r/k$, $r \in (R - R^{1/3},R + R^{1/3})$ and $3R/4 \leq k \leq 4R/3$, say, one has
\begin{align*} \int_{R - R^{1/3}}^{R + R^{1/3}} |J_{k}(r)| \, dr & = \int_{R - R^{1/3}}^{R + R^{1/3}} \left| J_{k} \left(k \cdot \tfrac{r}{k} \right) \right| \, dr\\
& \geq \int_{R - R^{1/3}}^{R + R^{1/3}} \left| \left(\frac{r}{4k} k^{2/3}\phi'\left(\frac{r}{k}\right)\right)^{-1/2}\Ai\left(-k^{2/3}\phi\left(\frac{r}{k}\right)\right)\right| \, dr - O(1/R)\\
& \sim \frac{1}{R^{1/3}} \int_{R - R^{1/3}}^{R + R^{1/3}} \left|\Ai\left(-k^{2/3}\phi\left(\frac{r}{k}\right) \right) \right| \, dr - O(1/R), \end{align*}
for large enough $R \geq 1$. It remains to show that
\begin{equation}\label{core} \int_{R - R^{1/3}}^{R + R^{1/3}} \left|\Ai\left(-k^{2/3}\phi\left(\frac{r}{k}\right) \right) \right| \, dr \gtrsim \frac{R^{1/3}}{2^{p/2}}, \end{equation}
assuming that $R - k \sim 4^{p}R^{1/3}$, and $3R/4 \leq k \leq R - R^{1/3}$. The latter condition ensures that $r/k \geq 1$ for all $r$ in the domain of integration, which means that equation (4) on \cite[p. 105]{E} is available: it gives that
\begin{displaymath} k \cdot \frac{2}{3}\left(\phi\left(\frac{r}{k}\right) \right)^{3/2} = \left(r^{2} - k^{2} \right)^{1/2} - k \cdot \cos^{-1} \left(\frac{k}{r}\right) =: f_{k}(r). \end{displaymath}
The derivative of $f_{k}$ is simply 
\begin{equation}\label{diff1} f_{k}'(r) = \frac{\sqrt{r^{2} - k^{2}}}{r} \sim \frac{\sqrt{r - k}}{R^{1/2}} \end{equation}
for $R/2 \leq k \leq r \leq 2R$, so that in particular
\begin{equation}\label{diff2} f_{k}'(r) \sim \frac{\sqrt{4^{p}R^{1/3}}}{R^{1/2}} = \frac{2^{p}}{R^{1/3}} \end{equation}
for $r \in (R - R^{1/3},R + R^{1/3})$ and $R - k \sim 4^{p}R^{1/3}$. Hence, by a change of variable,
\begin{align*} \int_{R - R^{1/3}}^{R + R^{1/3}} \left|\Ai\left(-k^{2/3}\phi\left(\frac{r}{k}\right) \right) \right| \, dr & = \int_{R - R^{1/3}}^{R + R^{1/3}} \left|\Ai\left(-\left(\frac{3}{2}\right)^{2/3}f_{k}(r)^{2/3} \right) \right| \, dr\\
& \gtrsim \frac{R^{1/3}}{2^{p}} \int_{R - R^{1/3}}^{R + R^{1/3}} \left|\Ai\left(-\left(\frac{3}{2}\right)^{2/3}f_{k}(r)^{2/3} \right) \right|f_{k}'(r) \, dr\\
& = \frac{R^{1/3}}{2^{p}} \int_{f_{k}(R - R^{1/3})}^{f_{k}(R + R^{1/3})} |\Ai(-cr^{2/3})| \, dr,  \end{align*}
where $c = (3/2)^{2/3}$. Write $a_{k} := f_{k}(R - R^{1/3})$ and $b_{k} := f_{k}(R + R^{1/3})$. Because $f_{k}(k) = 0$, one has (using \eqref{diff1})
\begin{align*} a_{k} & = \int_{k}^{R - R^{1/3}} f_{k}'(r) \, dr \sim \frac{1}{R^{1/2}} \int_{k}^{R - R^{1/3}} \sqrt{r - k} \, dr\\
& \sim \frac{(R - R^{1/3} - k)^{3/2}}{R^{1/2}} \sim \frac{(4^{p}R^{1/3})^{3/2}}{R^{1/2}} = 2^{3p}, \end{align*}
and $b_{k} - a_{k} \sim 2^{p}$, using \eqref{diff2}. Consequently, by definition of $\Ai$,
\begin{displaymath} \int_{a_{k}}^{b_{k}} |\Ai(-cr^{2/3})| \, dr = \int_{a_{k}}^{b_{k}} \frac{\sqrt{r^{2/3}}}{3} \left|J_{\tfrac{1}{3}}(r) + J_{-\tfrac{1}{3}}(r)\right| \, dr \sim 2^{p} \int_{a_{k}}^{b_{k}} \left|J_{\tfrac{1}{3}}(r) + J_{-\tfrac{1}{3}}(r)\right| \, dr. \end{displaymath}
Finally, one can has the following well-known asymptotic expansion for low-order Bessel functions:
\begin{displaymath} J_{\alpha}(r) = \sqrt{\frac{2}{\pi r}}\left(\cos\left(r - \frac{\alpha \pi}{2} - \frac{\pi}{4}\right) + O(1/r) \right). \end{displaymath}
In particular,
\begin{displaymath} \left|J_{\tfrac{1}{3}}(r) + J_{-\tfrac{1}{3}}(r)\right| \gtrsim \frac{1}{\sqrt{r}}\left|\cos\left(r - \frac{5\pi}{12}\right) + \cos\left(r - \frac{\pi}{12}\right)\right| - O(r^{-3/2}), \end{displaymath}
which, combined with the previous estimates, gives
\begin{align*} \int_{R - R^{1/3}}^{R + R^{1/3}} & \left|\Ai\left(-k^{2/3}\phi\left(\frac{r}{k}\right) \right) \right| \, dr\\
& \gtrsim \frac{R^{1/3}}{2^{3p/2}}\int_{a_{k}}^{b_{k}}\left|\cos\left(r - \frac{5\pi}{12}\right) + \cos\left(r - \frac{\pi}{12}\right)\right| - O(2^{-9p/4}). \end{align*}
Finally, it is clear that the last integral is $\sim b_{k} - a_{k} \sim 2^{p}$. This proves \eqref{core} and the lemma. \end{proof}

Now, the construction of the function $g$ can begin. Since the theorem claims that a suitable $g$ can be constructed for \textbf{any} $1/2$-dimensional measure $\mu$, it is natural that $g$ should somehow be derived from the measure itself. For the time being, it is convenient to think that $\mu$ is supported on $[0,1] \subset \R$ instead of $S^{1}$. Consider the following Littlewood-Paley decomposition of $\mu$. Let $\phi \colon \R \to [0,1]$ be a smooth radially decreasing function with $\phi(\xi) = 1$ for $|\xi| \leq R^{1/3}/2$ and $\phi(\xi) = 0$ for $|\xi| \geq R^{1/3}$. Moreover, choose $\phi$ so that $\phi = \hat{\eta}$ for some function $\eta \colon \R \to \R$ with $\eta(t) \sim R^{1/3}$ for $|t| \leq cR^{-1/3}$, and
\begin{equation}\label{form1} \eta(t) \lesssim_{N} R^{1/3} \sum_{k = 1}^{\infty} \frac{\chi_{\{| \cdot | \leq 2^{k}R^{-1/3}\}}}{2^{kN}}, \qquad N \in \N. \end{equation}
As usual, define
\begin{displaymath} \eta_{4^{-p}}(t) := 4^{p}\eta(4^{p}t). \end{displaymath}
Here are some standard estimates for $\mu \ast \eta_{4^{-p}}$:
\begin{lemma}\label{convolutionEstimates} Suppose that $\mu$ is an $s$-Ahlfors-David regular measure on $[0,1]$. Then
\begin{displaymath} \mu \ast \eta_{4^{-p}}(t) \gtrsim R^{(1-s)/3}4^{p(1-s)}, \qquad t \in \spt \mu, \end{displaymath}
and
\begin{displaymath} \mu \ast \eta_{4^{-p}}(t) \lesssim R^{(1 - s)/3}4^{p(1 - s)}, \qquad t \in \R. \end{displaymath}
\end{lemma}

\begin{proof} To obtain the lower bound, use 
\begin{displaymath} \mu(B(t,cR^{-1/3}4^{-p})) \gtrsim R^{-s/3}4^{-ps}, \end{displaymath}
combined with the fact that $\eta_{4^{-p}}(t) \sim R^{1/3}4^{p}$ for $|t| \leq cR^{-1/3}4^{-p}$. The upper bound follows from \eqref{form1} (with $N = 2$) and
\begin{displaymath} \mu(B(t,2^{k}R^{-1/3}4^{-p})) \lesssim 2^{ks}4^{-ps}R^{-s/3}, \end{displaymath}
which is automatically valid for all $t \in \R$ (and not just $t \in \spt \mu$). \end{proof}

The lemma has the useful corollary that for some large enough constant $K \geq 1$,
\begin{equation}\label{form2} \Delta_{p}(\mu)(t) := \mu \ast (\eta_{4^{-p - K}} - \eta_{4^{-p}})(t) \gtrsim R^{(1-s)/3}4^{p(1-s)}, \qquad t \in \spt \mu. \end{equation}
Observe that the Fourier coefficients of $\Delta_{p}(\mu)$ are supported on 
\begin{displaymath} D_{p} := D_{p}^{l} \cup D_{p}^{r} := -[cR^{1/3}4^{p},CR^{1/3}4^{p}] \cup [cR^{1/3}4^{p},CR^{1/3}4^{p}] \end{displaymath}
for some constants $c,C > 0$, depending only on $K$. The next step is to choose a collection of $\sim \log R$ disjoint sets $D_{p}$, with $CR^{1/3}4^{p} \leq R$, with $p$ so large that Lemma \ref{average} is applicable. Denote the family of these sets by $\calD$. Given $D_{p} \in \calD$, define the Fourier coefficients of $g_{r}$ on $R + D_{p}^{l}$ as follows:
\begin{displaymath} a_{k}(r) := \widehat{g_{r}}(k) := \frac{\widehat{\Delta_p(\mu)}(k - R)\sgn J_{k}(r)}{2^{p} A_{k}(R)}, \qquad k \in R + D_{p}^{l}. \end{displaymath}
Observe that the sets $R + D_{p}^{l}$ are disjoint for various $D_{p} \in \calD$, and the Fourier coefficients of $g_{r}$ are only defined "to the left from $R$" (the reason for this is that the lower bound in Lemma \ref{average} is only valid in the region $k < R$). This completes the definition of $g_{r}$, so all other Fourier coefficients are simply zero. Next, assume that $R \in \N$, and observe that
\begin{align*} \frac{1}{R^{1/3}} & \int_{R - R^{1/3}}^{R + R^{1/3}} \int_{S^{1}} |\widehat{g_{r}d\sigma}(rx)|^{2} \, d\mu x \, dr\\
& = \int_{0}^{1} \frac{1}{R^{1/3}} \int_{R - R^{1/3}}^{R + R^{1/3}} \left|\sum_{D_{p} \in \calD} \sum_{k \in R + D_{p}^{l}} a_{k}J_{k}(r)e^{2\pi i (k - R) \theta} \right|^{2} \, dr \, d\mu \theta\\
& \geq \int_{0}^{1} \left| \sum_{D_{p} \in \calD} \sum_{k \in R + D_{p}^{l}} \frac{1}{R^{1/3}} \int_{R - R^{1/3}}^{R + R^{1/3}} \frac{\widehat{\Delta_p(\mu)}(k - R)|J_{k}(r)|}{2^{p} A_{k}(R)}e^{2\pi i(k - R)\theta}\right|^{2} \, d\mu \theta\\
& = \int_{0}^{1} \left| \sum_{D_{p} \in \calD} \frac{1}{2^{p}} \sum_{k \in \calD_{p}^{l}} \widehat{\Delta_{p}(\mu)}(k)e^{2\pi i k\theta}\right|^{2} \, d\mu \theta\\
& \geq \int_{0}^{1} \left(\sum_{D_{p} \in \calD} \frac{1}{2^{p}} \Rea \sum_{k \in \calD_{p}^{l}} \widehat{\Delta_{p}(\mu)}(k)e^{2\pi i k\theta} \right)^{2} \, d\mu \theta.   \end{align*} 
Because the function $\Delta_{p}(\mu)$ is real-valued, one has
\begin{align*} \Rea \widehat{\Delta_{p}(\mu)}(k)2^{2\pi ik\theta} = \frac{\widehat{\Delta_{p}(\mu)}(k)e^{2\pi k\theta} + \widehat{\Delta_{p}(\mu)}(-k)e^{2\pi i(-k)\theta}}{2},  \end{align*}
which implies that
\begin{displaymath} \Rea \sum_{k \in D_{p}^{l}} \widehat{\Delta_{p}(\mu)}(k)e^{2\pi i k\theta} = \frac{1}{2} \sum_{k \in D_{p}} \widehat{\Delta_{p}(\mu)}(k)e^{2\pi i k\theta} = \frac{\Delta_{p}(\mu)(\theta)}{2}, \end{displaymath}
recalling that the Fourier coefficients of $\Delta_{p}(\mu)$ are supported on $D_{p}$. For every $D_{p} \in \calD$, the lower bound \eqref{form2} with $s = 1/2$ now gives
\begin{align*} \frac{1}{R^{1/3}} & \int_{R - R^{1/3}}^{R + R^{1/3}} \int_{S^{1}} |\widehat{g_{r}d\sigma}(rx)|^{2} \, d\mu x \, dr\\
& \gtrsim \int_{0}^{1} \left(\sum_{D_{p} \in \calD} \frac{1}{2^{p}} \cdot R^{1/6}2^{p} \right)^{2} \, d\mu \theta \gtrsim R^{1/3} \cdot (\log R)^{2}. \end{align*}
In particular, there exists a radius $r \in (R - R^{1/3},R + R^{1/3})$ such that
\begin{equation}\label{form3}  \int_{S^{1}} |\widehat{g_{r}d\sigma}(rx)|^{2} \, d\mu x \gtrsim R^{1/3} \cdot (\log R)^{2}. \end{equation}
It remains to give a uniform upper bound for the $L^{2}$-norms of the functions $g_{r}$. By Plancherel and Lemma \ref{average},
\begin{align*} \|g_{r}\|_{2}^{2} = \sum_{D_{p} \in \calD} \sum_{k \in R + D_{p}^{l}} \frac{|\widehat{\Delta_{p}(\mu)}(k - R)|^{2}}{4^{p}A_{k}(R)^{2}} \lesssim R^{2/3} \sum_{D_{p} \in \calD} \frac{1}{2^{p}} \sum_{k \in \calD_{p}} |\widehat{\Delta_{p}(\mu)}(k)|^{2}, \end{align*}
because $R - k \sim 4^{p}R^{1/3}$ for $k \in R + D_{p}^{l}$ and $p$ was chosen large enough to begin with. The inner sum on the right hand side is the squared $L^{2}$-norm of $\Delta_{p}(\mu) = \mu \ast (\eta_{4^{-p - K}} - \eta_{4^{-p}})$, which can be bounded by estimating separately the $L^{2}$-norms of $\mu \ast \eta_{4^{-p - K}}$ and $\mu \ast \eta_{4^{-p}}$. For instance, applying the upper bound from Lemma \ref{convolutionEstimates}, one finds that
\begin{displaymath} \|\mu \ast \eta_{4^{-p}}\|_{2}^{2} \lesssim R^{1/6} \cdot 2^{p} \int (\mu \ast \eta_{4^{-p}})(t) \, dt = R^{1/6} \cdot 2^{p}, \end{displaymath}
since $\int \eta = 1$ and $\mu$ is a probability measure. Finally,
\begin{displaymath} \|g_{r}\|_{2}^{2} \lesssim R^{2/3} \sum_{D_{p} \in \calD} R^{1/6} \sim R^{5/6} \cdot \log R, \end{displaymath} 
which in combination with \eqref{form3} shows that 
\begin{displaymath} \int_{S^{1}} |\widehat{g_{r} d\sigma}(rx)|^{2} \, d\mu x \gtrsim \frac{\log R}{R^{1/2}} \cdot R^{5/6} \cdot \log R \gtrsim \frac{\log R}{R^{1/2}} \cdot \|g_{r}\|_{2}^{2}. \end{displaymath}

\subsection{Sharpness of the bounds for $s \neq \tfrac{1}{2}$} The sharpness of the bound in the case $s > 1/2$ is easily seen using a Knapp type example; more precisely, take $g(x) = e^{2\pi a \cdot x}\chi_{J}$, where $J$ is an arc of length $\ell(J) \sim R^{-1/3}$ and $a \in \R^{2}$ will be chosen momentarily. Then $\|g\|_{2}^{2} \sim \ell(J) = R^{-1/3}$, and $|\widehat{gd\sigma}(x)| \gtrsim R^{-1/3}$ for $x \in T$, where $T$ is a rectangle with dimensions $R^{1/3} \times R^{2/3}$. Choosing $a$ appropriately, this rectangle can be placed so that the intersection $(R \cdot S^{1}) \cap T$ is an arc of length $\sim R^{2/3}$, where $R \cdot S^{1} = \{|x| = R\}$. In particular, 
\begin{displaymath} R \cdot \int_{S^{1}} |\widehat{g d\sigma}(Rx)|^{2} \, d\mu x \gtrsim R^{1/3} \cdot \mu(\{x : Rx \in T\}) \gtrsim R^{(1 - s)/3} \sim R^{(2 - s)/3} \cdot \|g\|_{2}^{2}, \end{displaymath}
where the right hand side is $\sim M_{R}(\mu) \cdot \|g\|_{2}^{2}$ under the assumption $s > 1/2$.

In the case $s < 1/2$, the easiest way (at this point, at least) is probably to review the proof of the case $s = 1/2$. The functions $g_{r}$ can be defined exactly as before, and using \eqref{form2} yields
\begin{displaymath} \int_{S^{1}} |\widehat{g_{r}d\sigma}(rx)|^{2} \, d\mu x \gtrsim R^{4/3 - 2s} \end{displaymath}
instead of \eqref{form3}, for some $r \in (R - R^{1/3},R + R^{1/3})$. On the other hand, applying Lemma \ref{convolutionEstimates} as above yields the uniform bound $\|g_{r}\|_{2}^{2} \lesssim R^{4/3 - s}$. Then, if $s < 1/2$, Proposition \ref{constants} implies that
\begin{displaymath} R \cdot \int_{S^{1}} |\widehat{g_{r}d\sigma}(rx)|^{2} \, d\mu x \gtrsim R^{1 - s}\|g_{r}\|_{2}^{2} \sim M_{R}(\mu)\|g\|_{2}^{2}, \end{displaymath}
completing the proof of Theorem \ref{main}.

\end{document}